\documentclass[12pt,reqno]{amsart}

\usepackage{amsmath,amssymb,amsthm}
\usepackage{graphicx}
\allowdisplaybreaks
\usepackage{xcolor}
\usepackage{fullpage}
\usepackage[linktoc=all,colorlinks,linkcolor=teal,citecolor=cyan]{hyperref}
\usepackage{cleveref}
\numberwithin{equation}{section}
\usepackage{comment}

\newcommand{\mb}[1]{\mathbb{#1}}
\newcommand{\mc}[1]{\mathcal{#1}}
\newcommand{\st}[1]{\substack{#1}}

\newtheorem{thm}{Theorem}[section]
\newtheorem{lemma}{Lemma}[section]
\newtheorem{prop}{Proposition}[section]

\title{Diophantine Approximation with Piatetski-Shapiro Primes}

\author{Stephan Baier}
\address{Stephan Baier,
Ramakrishna Mission Vivekananda Educational and Research Institute, Department of Mathematics, G. T. Road, PO Belur Math, Howrah, West Bengal 711202, India}
\email{stephanbaier2017@gmail.com}
\author{Habibur Rahaman}
\address{Habibur Rahaman, Indian Institute of Science Education \& Research Kolkata,
Department of Mathematics and Statistics, Mohanpur, West Bengal 741246, India}
\email{hr21rs044@iiserkol.ac.in}

\begin{document}
\begin{abstract} We prove that for every irrational number $\alpha$, real number $\beta$, real number $c$ satisfying $1<c<9/8$ and positive real number $\theta$ satisfying $\theta<(9/c-8)/10$, there
exist infinitely many primes of the form $p=\left[n^c\right]$ with $n\in \mathbb{N}$ such that $||\alpha p+\beta||<p^{-\theta}$.
\end{abstract}

\maketitle


\section{Introduction}\label{sec_into_PSprimes}
In this article, we consider Diophantine approximation with denominators which are restricted to Piatetski-Shapiro primes.  These are prime numbers $p$ of the form $\left[n^c\right]$, where $c>1$ is a fixed real number, $n$ runs over the positive integers, and $[x]$ denotes the integral part of $x\in \mathbb{R}$. One may conjecture that, given any non-integer $c>1$, there exist infinitely many such primes. Their investigation was initiated by Piatetski-Shapiro \cite{Pia} who demonstrated their infinitude for $1<c<12/11=1.0909...$. This range has been widened by many authors. The latest record is due to Rivat and Wu \cite{RiW} who obtained a range of $1<c<243/205=1.1853...$. Progress on Piatetski-Shapiro primes measures progress on exponential sums and sieve methods.

The Dirichlet approximation theorem, a cornerstone in Diophantine approximation, implies that for any irrational number $\alpha$, there exist infinitely many positive integers $q$ such that $||\alpha q||<q^{-1}$, where $||x||$ is the distance of $x\in \mathbb{R}$ to the nearest integer. Interesting problems arise when the $q$'s are restricted to a sparse subset of the natural numbers. Many authors have considered the question for which $\theta>0$ one can show the infinitude of {\it primes} $p$ such that $||\alpha p||<p^{-\theta}$. It may be expected that an exponent of $\theta=1-\varepsilon$ is admissible. The first author to consider this problem was Vinogradov who proved that $\theta=1/5-\varepsilon$ is admissible \cite{Vin}. The current record is due to Matom\"aki \cite{Mat} who obtained an exponent of $\theta=1/3-\varepsilon$.  Harman's sieve has become a standard tool for the investigation of this problem (for details on Harman's sieve see \cite{prime_detective_sieve}). 

Dimitrov considered a hybrid problem, restricting the set of primes $p$ in the above Diophantine approximation problem to Piatetski-Shapiro primes. He proved the following in \cite{Dimitrov}.

\begin{thm}[Dimitrov]\label{Dimres}
		Fix an irrational number $\alpha$, a real number $\beta$, a real number $c$ satisfying $1<c<12/11$ and a positive real number $\theta$ satisfying $\theta<(12/c-11)/26$.
Then there exist infinitely many primes of the form $p=\left[n^c\right]$ with $n\in \mathbb{N}$ such that $||\alpha p+\beta||<p^{-\theta}$.
\end{thm}

In this article, we improve Dimitrov's result, widening the $c$- and $\theta$-ranges. We will establish the following. 

\begin{thm}\label{main_thm}
		Fix an irrational number $\alpha$, a real number $\beta$, a real number $c$ satisfying $1<c<9/8$ and a positive real number $\theta$ satisfying $\theta<(9/c-8)/10$.
Then there exist infinitely many primes of the form $p=\left[n^c\right]$ with $n\in \mathbb{N}$ such that $||\alpha p+\beta||<p^{-\theta}$.
\end{thm}
	
Whereas Dimitrov used Vaughan's identity to obtain his result, we here apply Harman's sieve. We also employ arguments due to Heath-Brown \cite{H-B} and  Balog and Friedlander \cite{Balog} to improve certain estimates of bilinear exponential sums. 
\subsection{Notation} In this article, we will use the following notations.
\begin{itemize}
\item We will denote the set of integers, positive integers and primes by $\mathbb{Z}$, $\mathbb{N}$ and $\mathbb{P}$, resprectively. 
    \item Expressions of the form ${f}(x)=O({g}(x))$, ${f}(x) \ll {g}(x)$, and ${g}(x) \gg {f}(x)$ signify that $|{f}(x)| \leq C|{g}(x)|$ for all sufficiently large $x$, where $C>0$ is an absolute constant. A subscript of the form $\ll_A$ means that the implied constant may depend on the parameter $A$. 
    \item For $x,y>0$, the notation $x\asymp y$ means that there are constants $C_2>C_1>0$ such that $C_1x\le y\le C_2x$.
    \item  The notation $x\sim X$ means that $X/2\le x<X$.
    \item For any real number $z$, we will write $e(z):=e^{2\pi iz}$.
    \item We will denote the divisor function by $\tau$, that is, for any $n\in\mb{N}$, $$\tau(n):=\sum_{d|n}1.$$
\end{itemize}
{\bf Acknowledgments.} The authors thank the anonymous referee for his valuable comments. The first-named author thanks the Ramakrishna Mission Vivekananda Educational and Research Institute for an excellent work environment. The research of the second-named author was supported by a Prime Minister Research Fellowship (PMRF ID- 0501972), funded by the Ministry of Education, Govt. of India.

\section{Application of Harman's sieve}
For $X\in \mathbb{N}$ define
\begin{equation} \label{Adef}
\mathcal{A}:=\left\{a\in \mathbb{N} : X/2\le a<X, \ ||\alpha a+\beta||<\Delta, \  || a^{\gamma}+2\delta||<\delta\right\} 
\end{equation}
where we set 
\begin{equation} \label{Deltadef}
\Delta:=X^{-\theta}.
\end{equation}
and 
\begin{equation} \label{deltadef}
\delta:= \frac{\gamma X^{\gamma-1}}{10}. 
\end{equation}
The bulk of our proof of Theorem \ref{main_thm} will consist of establishing the following.

\begin{lemma}\label{bulk} There exist infinitely many positive integers $X$ such that
\begin{equation} \label{newgoal}
\sharp(\mathcal{A}\cap \mathbb{P})>0, 
\end{equation}
where the set $\mathcal{A}$ is defined as in \eqref{Adef}, depending on $X$. 
\end{lemma}

To establish this result, we shall employ Harman's sieve. Set
\begin{equation} \label{Bdef}
\mathcal{B}:=[X/2,X)\cap \mathbb{N}.
\end{equation}
The idea is to compare the cardinalities $\sharp(\mathcal{A}\cap \mathbb{P})$ and $\sharp(\mathcal{B}\cap \mathbb{P})$ via a comparison of bilinear sums associated to the sets $\mathcal{A}$ and $\mathcal{B}$.  
Precisely, we have the following. 

\begin{prop}[Harman] \label{Thm2}
Let $X\ge 4$ and $\mathcal{A}\subseteq \mathcal{B}:=[X/2,X)\cap \mathbb{N}$.  Suppose that for some $\eta>0$ and some $\lambda>0$ we have, for all sequences $a_m$ and $b_n$ of non-negative real  numbers satisfying 
		\begin{align} \label{ambnsizes}
			a_m\leq \tau(m), \ b_n\leq \tau(n)
		\end{align}
($\tau(n)$ here denotes the number of divisors of $n$) that
		\begin{align}\label{type_i_thm2}
			\sum_{\st{mn\in\mc{A}\\m\leq X^{15/22}}}a_m=\lambda \sum_{\st{mn\in\mc{B}\\m\leq X^{15/22}}}a_m+O\left(\lambda X^{1-\eta}\right)
		\end{align}
and 
\begin{align}\label{type_ii_thm2}
			\sum_{\st{mn\in\mc{A}\\X^{7/22}\leq m\leq X^{8/22}}}a_mb_n=\lambda \sum_{\st{mn\in\mc{B}\\ X^{7/22}\leq m\leq X^{8/22}}}a_mb_n +O\left(\lambda X^{1-\eta}\right).
		\end{align}
		Then the inequality 
		\begin{align*}
	          \sharp(\mc{A}\cap\mathbb{P})>\frac{\lambda}{10}\cdot \sharp(\mc{B}\cap\mathbb{P})
		\end{align*}
holds, provided that $X$ is large enough.
	\end{prop}

\begin{proof}
Apply \cite[Theorem 2]{Harman} with $\theta=7/22$ and note that if $X\ge 4$, then $\sharp(\mc{A}\cap\mathbb{P})=S(\mathcal{A},X^{1/2})$ and $\sharp(\mc{B}\cap\mathbb{P})=S(\mathcal{B},X^{1/2})$, where the sieve function $S(\mathcal{C},z)$ is defined as in \cite{Harman}. Also note that the sequences $a_m$ and $b_n$ can be confined to non-negative real numbers. (In the original statement, $a_m$ and $b_n$ were assumed to be complex. Extracting the real and imaginary parts of $a_m$ and $b_n$, it suffices to assume that they are real. Then dividing the sums over $m$ and $n$ into subsums according to the signs of $a_m$ and $b_n$, it suffices to assume that they are non-negative.)  
\end{proof}

Following usual terminology, we will refer to the sums in \eqref{type_i_thm2} as type I sums and the sums in \eqref{type_ii_thm2} as type II sums.
To establish Lemma \ref{bulk}, it now suffices to establish the bounds \eqref{type_i_thm2} and \eqref{type_ii_thm2} above for $X\in \mathcal{X}$, where $\mathcal{X}$ is an infinite subset of  $\mathbb{N}$.  This set is constructed as follows. By the Dirichlet approximation theorem, mentioned in Section~\ref{sec_into_PSprimes}, there exist infinitely many positive integers $q$ such that 
\begin{equation} \label{Diocond}
\left|\alpha - \frac{a}{q} \right|<q^{-2} \quad \mbox{for some } a\in \mathbb{Z} \mbox{ with } (a,q)=1.
\end{equation}
We shall take $\mathcal{X}$ to be the set of all positive integers $X$ such that
\begin{equation} \label{Xqcondi}
X^{2\theta+10\eta}\le q\le X^{1-\theta-10\eta}
\end{equation}
for some $q\in \mathbb{N}$ satisfying \eqref{Diocond}, where $\eta$ is a suitably small positive number. After having established \eqref{type_i_thm2} and \eqref{type_ii_thm2} for all $X\in \mathcal{X}$, Lemma \ref{bulk} follows from Proposition \ref{Thm2}. In the remainder of this article, we assume that $X\in \mathcal{X}$.

We conclude this section by deducing our main result, Theorem \ref{main_thm}, from Lemma \ref{bulk}.  Let $\gamma=1/c$ and $p\in \mathbb{O}$. Obviously, $p=\left[n^c\right]$ for some $n\in \mathbb{N}$ if and only if 
\begin{equation} \label{equi}
p^{\gamma}\le n< (p+1)^{\gamma} \quad \mbox{for some } n\in \mathbb{N}.
\end{equation}
Moreover, $(p+1)^{\gamma}=p^{\gamma}+\gamma p^{\gamma-1}+O\left(p^{\gamma-2}\right)$.  Hence, if 
$\delta$ is defined as in \eqref{deltadef}, then 
\begin{equation} \label{equi2}
p^{\gamma}\in (n-3\delta,n-\delta) \quad \mbox{for some } n\in \mathbb{N}
\end{equation}
implies \eqref{equi}, provided that $X/2\le p<X$ with $X$ large enough. Clearly, the condition \eqref{equi2} is equivalent to the inequality
\begin{equation} \label{equi3}
|| p^{\gamma}+2\delta||<\delta.
\end{equation} 
Therefore, Theorem \ref{main_thm} follows from Lemma \ref{bulk}.

\section{Fourier analysis}
We shall detect the conditions $||\alpha a+\beta||<\Delta$ and  $|| a^{\gamma}+2\delta||<\delta$ in \eqref{Adef} using the following Fourier analytic device.

\begin{lemma}[Harman] \label{chi(x)}
		Let $\xi \in(0,1)$ and $K$ be any positive integer. Define
		\begin{align*}
			\chi_{\xi}(x)=\begin{cases}
				1 , \ \text{ if } ||x||<\xi,\\
				0, \ \text{otherwise.}
			\end{cases}
		\end{align*}
		Then there are sequences $c^-_k$ and $c^+_k$ of complex numbers such that  
		\begin{align*}
			|c^{\pm}_k|\leq \min\left\{2\xi+\frac{1}{K+1},\  \frac{3}{2|k|}\right\}
		\end{align*} 
for all $k\in \mathbb{N}$ and
 \begin{align*}
			\chi_{\xi}^-(x) \leq \chi_{\xi}(x)\leq \chi_{\xi}^+(x)
		\end{align*}
for all $x\in \mathbb{R}$, where 
		\begin{align*}
\chi_{\xi}^-(x):=2\xi-\frac{1}{K+1}-\sum_{0<|k|\leq K}c^{-}_ke(kx)\quad \mbox{and} \quad \chi_{\xi}^+(x):= 2\xi+\frac{1}{K+1}+\sum_{0<|k|\leq K}c^{+}_ke(kx).
		\end{align*}
\end{lemma}
\begin{proof}
This is \cite[Lemma~2.1]{prime_detective_sieve}. 
\end{proof}

We recall the definitions of $\Delta$ and $\delta$ in \eqref{Deltadef} and \eqref{deltadef}, fix an arbitrarily small $\eta>0$ and set
\begin{equation} \label{LHdef}
L:=\left[\Delta^{-1}X^{\eta}\right]=\left[X^{\theta+\eta}\right] \quad \mbox{and} \quad  H:=\left[\delta^{-1}X^{\eta}\right]=\left[10X^{1-\gamma+\eta}/\gamma\right].
\end{equation}
Then Lemma \ref{chi(x)} above with $\xi=\Delta,\delta$ and $K=L,H$ produces functions $\chi_{\Delta}^{\pm}$ and $\chi_{\delta}^{\pm}$ such that  
$$
\chi_{\Delta}^-(x) \leq \chi_{\Delta}(x)\leq \chi_{\Delta}^+(x)
$$
and
$$
\chi_{\delta}^-(y) \leq \chi_{\delta}(y)\leq \chi_{\delta}^+(y)
$$
for all $x,y\in \mathbb{R}$. We want to use these functions to bound the {\it product} 
$ \chi_{\Delta}(x)\chi_{\delta}(y)$ from below and above. To this end, we use the following observation from \cite{Harman2}: For all $x,y\in \mathbb{R}$, we have 
$$
\Xi^-(x,y)\le \chi_{\Delta}(x)\chi_{\delta}(y) \le \Xi^+(x,y),
$$
where
$$
\Xi^-(x,y):=\chi_{\Delta}^-(x)\chi_{\delta}^+(y)+\chi_{\Delta}^+(x)\chi_{\delta}^-(y)-\chi_{\Delta}^+(x)\chi_{\delta}^+(y) \quad \mbox{and} \quad \Xi^+(x,y):=\chi_{\Delta}^+(x)\chi_{\delta}^+(y).
$$  
Recalling the definitions of the sets $\mathcal{A}$ and $\mathcal{B}$ in \eqref{Adef} and \eqref{Bdef}, and keeping the condition of non-negativity of $a_m$ and $b_n$ in Proposition \ref{Thm2} in mind, it follows that the sums on the left-hand sides of \eqref{type_i_thm2} and \eqref{type_ii_thm2} are bounded from below and above by  
\begin{equation*}
\sum_{\st{mn\in\mc{B}\\m\leq X^{15/22}}} a_m\Xi^-\left(\alpha mn+\beta, (mn)^{\gamma}+2\delta\right) \le  \sum_{\st{mn\in\mc{A}\\m\leq X^{15/22}}}a_m\le  \sum_{\st{mn\in\mc{B}\\m\leq X^{15/22}}} a_m\Xi^{+}\left(\alpha mn+\beta, (mn)^{\gamma}+2\delta\right)
\end{equation*} 
and 
\begin{equation*}
\begin{split}
\sum_{\st{mn\in\mc{B}\\X^{7/22}\leq m\leq X^{8/22}}} a_mb_n\Xi^{-}\left(\alpha mn+\beta, (mn)^{\gamma}+2\delta\right)
\le & \sum_{\st{mn\in\mc{A}\\X^{7/22}\leq m\leq X^{8/22}}}a_mb_n \\ \le &
\sum_{\st{mn\in\mc{B}\\X^{7/22}\leq m\leq X^{8/22}}} a_mb_n\Xi^{+}\left(\alpha mn+\beta, (mn)^{\gamma}+2\delta\right).
\end{split}
\end{equation*}
Now using the definitions of $\Xi^-(x,y)$ and $\Xi^+(x,y)$, replacing $\chi_{\Delta}^{\pm}(x)$ and $\chi_{\delta}^{\pm}(y)$ by the relevant trigonometrical polynomials from Lemma \ref{chi(x)}, multiplying out, and noting the condition \eqref{ambnsizes} and the well-known bound $\tau(n)\ll_{\varepsilon} n^{\varepsilon}$ for every $\varepsilon>0$, we deduce that
\begin{align*}
			\sum_{\st{mn\in\mc{A}\\m\leq X^{15/22}}}a_m=\lambda \sum_{\st{mn\in\mc{B}\\m\leq X^{15/22}}}a_m+O\left(\lambda X^{1-\eta}+\Sigma_I\right)
		\end{align*}
and 
\begin{align*}
			\sum_{\st{mn\in\mc{A}\\X^{7/22}\leq m\leq X^{8/22}}}a_mb_n=\lambda \sum_{\st{mn\in\mc{B}\\ X^{7/22}\leq m\leq X^{8/22}}}a_mb_n +O\left(\lambda X^{1-\eta}+\Sigma_{II}\right)
		\end{align*}
with
\begin{equation} \label{lambdadef}
\lambda:=4\Delta\delta,
\end{equation}
where the error term $\Sigma_1$ is a sum of expressions of the form
\begin{equation} \label{S1def}
\mathcal{S}_1:=\delta \sum_{\st{mn\sim X\\m\le X^{15/22}}}\sum_{0<|l|\leq L}a_mc_le(\alpha lmn),
\end{equation}
\begin{equation} \label{S2def}
\mathcal{S}_2:=\Delta\sum_{\st{mn\sim X\\m\le X^{15/22}}} \sum_{0<|h|\leq H}a_m d_he\left(h(mn)^{\gamma}\right),
\end{equation}
\begin{equation} \label{S3def}
\mathcal{S}_3:= \sum_{\st{mn\sim X\\m\le X^{15/22}}}\sum_{0<|l|\leq L}\sum_{0<|h|\leq H}  a_mc_ld_he\left(\alpha lmn+h(mn)^{\gamma}\right),
\end{equation}
and the error term $\Sigma_2$ is a sum of expressions of the form
\begin{equation} \label{T1def}
\mathcal{T}_1:=\delta \sum_{\st{mn\sim X\\X^{7/22}\leq m\leq X^{8/22}}}\sum_{0<|l|\leq L} a_mb_nc_le(\alpha lmn),
\end{equation}
\begin{equation} \label{T2def}
\mathcal{T}_2:=\Delta \sum_{\st{mn\sim X\\X^{7/22}\leq m\leq X^{8/22}}}\sum_{0<|h|\leq H} a_mb_n d_he\left(h(mn)^{\gamma}\right),
\end{equation}
\begin{equation} \label{T3def}
\mathcal{T}_3:=\sum_{\st{mn\sim X\\X^{7/22}\leq m\leq X^{8/22}}} \sum_{0<|l|\leq L} \sum_{0<|h|\leq H}  a_mb_nc_ld_he\left(\alpha lmn+h(mn)^{\gamma}\right),
\end{equation}
the coefficients satisfying the bounds
\begin{align*}
			a_m\ll m^{\varepsilon}, \quad b_n\ll n^{\varepsilon},\quad c_l\ll \Delta,\quad d_h\ll \delta.
		\end{align*} 
To establish Theorem \ref{main_thm}, it now suffices to prove that 
\begin{equation*} 
\mathcal{S}_i, \mathcal{T}_i\ll \lambda X^{1-\eta}
\end{equation*} 
for $i=1,2,3$ and a suitably small $\eta>0$. This is the content of the remainder of this paper. 

\section{Tailoring the type I and II sums}
It will be advantageous to tailor the expressions $\mathcal{S}_i$, $\mathcal{T}_i$ above. Breaking the sums on the right-hand sides of \eqref{S1def}-\eqref{T3def} into dyadic subsums with $m\sim M$, $n\sim N$, $|l|\sim U$, $|h|\sim V$ and scaling the coefficients, setting
$$
a_m^{\ast}:=\frac{a_m}{m^{\varepsilon}}, \quad b_n^{\ast}:=\frac{b_n}{n^{\varepsilon}}, \quad c_l^{\ast}:=
\frac{c_l}{\Delta}, \quad d_h^{\ast}:=\frac{d_h}{\delta},
$$
we reduce these sums to $O(\log^4 X)$ expressions of the form
\begin{equation} \label{S1astdef}
\mathcal{S}_1^{\ast}:=\sum_{\substack{m\sim M\\ n\sim N\\ mn\sim X}} \sum_{|l|\sim U}a_m^{\ast}c_l^{\ast}e(\alpha lmn),
\end{equation}
\begin{equation} \label{S2astdef}
\mathcal{S}_2^{\ast}:=\sum_{\substack{m\sim M\\ n\sim N\\ mn\sim X}} \sum_{|h|\sim V} a_m^{\ast}d_h^{\ast}e\left(h(mn)^{\gamma}\right),
\end{equation}
\begin{equation} \label{S3astdef}
\mathcal{S}_3^{\ast}:=\sum_{\substack{m\sim M\\ n\sim N\\ mn\sim X}} \sum_{|l|\sim U} \sum_{|h|\sim V}  a_m^{\ast}c_l^{\ast}d_h^{\ast}e\left(\alpha lmn+h(mn)^{\gamma}\right),
\end{equation}
\begin{equation} \label{T1astdef}
\mathcal{T}_1^{\ast}:=\sum_{\substack{m\sim M\\ n\sim N\\ mn\sim X}} \sum_{|l|\sim U} a_m^{\ast}b_n^{\ast}c_l^{\ast}e(\alpha lmn),
\end{equation}
\begin{equation} \label{T2astdef}
\mathcal{T}_2^{\ast}:=\sum_{\substack{m\sim M\\ n\sim N\\ mn\sim X}}\sum_{|h|\sim V}  a_m^{\ast}b_n^{\ast}d_h^{\ast}e\left(h(mn)^{\gamma}\right),
\end{equation}
\begin{equation} \label{T3astdef}
\mathcal{T}_3^{\ast}:=  \sum_{\substack{m\sim M\\ n\sim N\\ mn\sim X}} \sum_{|l|\sim U} \sum_{|h|\sim V}a_m^{\ast}b_n^{\ast}c_l^{\ast}d_h^{\ast}e\left(\alpha lmn+h(mn)^{\gamma}\right),
\end{equation}
where the coefficients $a_m^{\ast}$, $b_n^{\ast}$, $c_l^{\ast}$, $d_h^{\ast}$ are complex numbers satisfying
\begin{equation*} 
a_m^{\ast}, b_n^{\ast}, c_l^{\ast}, d_h^{\ast}\ll 1.
\end{equation*}	
Now it remains to prove that
\begin{equation} \label{Sitarget}
\mathcal{S}_i^{\ast}\ll  X^{1-2\eta} \mbox{ for } i=1,2,3,\ M\le X^{15/22},\ U\le L,\ V\le H  
\end{equation}
and 
\begin{equation} \label{Titarget}
\mathcal{T}_i^{\ast}\ll X^{1-2\eta} \mbox{ for } i=1,2,3,\ X^{7/22}\le M\le X^{8/22}, \ N\asymp X/M,\ U\le L, V\le H
\end{equation}
for $\eta>0$ small enough. 

\section{Estimations of $\mathcal{S}_1^{\ast}$ and $\mathcal{T}_1^{\ast}$} 
Recall from section 2 that $X\in \mathcal{X}$. Therefore, \eqref{Diocond} and \eqref{Xqcondi} hold for this $X$ and suitable $a,q\in \mathbb{Z}$. To estimate $\mathcal{S}_1^{\ast}$ and $\mathcal{T}_1^{\ast}$, defined in \eqref{S1astdef} and \eqref{T1astdef}, we use the following standard bounds for sums involving linear exponential terms.  
\begin{lemma}\label{lemma_type_I_II_exponential_sum}
		Let $K,N\ge 1$ and $\alpha_k,\beta_n$ be any sequences of complex numbers. Assume that \eqref{Diocond} holds. Then we have
		\begin{equation} \label{linear}
	          \sum\limits_{n\asymp N} e(\alpha kn)\ll \min\left\{N,||\alpha k||^{-1}\right\} \mbox{ for all } k\in \mathbb{N},
\end{equation}
\begin{equation} \label{sumoverlinear}
	          \sum\limits_{k\asymp K} \min\left\{N,||\alpha k||^{-1}\right\}\ll \left(\frac{KN}{q}+K+q\right)(\log 2KNq)
		\end{equation}
		and 
		\begin{equation}\label{type_II_exponential}
\sum_{\substack{k\asymp K\\ n\asymp N\\ kn\asymp KN}}\alpha_k\beta_ne(\alpha mn)
\ll \left(\sum_{k\asymp K}|\alpha_k|^2\sum_{n\asymp N}|\beta_n|^2\right)^{1/2}\left(\frac{KN}{q}+K+N+q\right)^{1/2}(\log 2 KNq)^{1/2}.
\end{equation}
\end{lemma}
		\begin{proof}
			See \cite[section 1.6]{prime_detective_sieve}.
		\end{proof}
Now recalling the inequalities
\begin{equation} \label{mainconds}
 \frac{8}{9}<\gamma=1/c<1 \quad \mbox{and} \quad 0<\theta<\frac{9\gamma-8}{10}<\frac{1}{10}
\end{equation}
from Theorem \ref{main_thm},
we are ready to prove the desired bounds for $\mathcal{S}_1^{\ast}$ and $\mathcal{T}_1^{\ast}$. Throughout the sequel, we assume that $\varepsilon$ is a fixed but arbitrarily small positive real number. 

\begin{lemma} \label{ST1esti} There is $\eta>0$ such that  $\mathcal{S}_1^{\ast},\mathcal{T}_1^{\ast}\ll  X^{1-2\eta}$.
\end{lemma}

\begin{proof}
Applying \eqref{linear} to the sum over $n$ on the right-hand side of \eqref{S1astdef}, we have 
$$
\mathcal{S}_1^{\ast}\ll \sum_{m\sim M} \sum_{l\sim U}a_m^{\ast}c_l^{\ast} \min\left\{\frac{X}{M},\frac{1}{||\alpha ml||}\right\}.
$$
(Here and below, we treat the cases $l>0$ and $l<0$ in a similar way.)
Writing $ml=k$, using the bound
\begin{equation} \label{alphakdef}
\alpha_k:=\sum_{\substack{m\sim M\\ l\sim U\\ ml=k}} a_m^{\ast}c_l^{\ast} \ll X^{\varepsilon}
\end{equation}
and applying \eqref{sumoverlinear} with $K:=MU$ and $N:=X/M$, it follows that
$$
\mathcal{S}_1^{\ast}\ll \left(\frac{UX}{q}+MU+q\right)X^{\varepsilon}.
$$ 
Recalling 
\begin{equation} \label{recall}
U\le L\ll X^{\theta+\eta}\quad \mbox{and} \quad X^{2\theta+10\eta}\le q\le X^{1-\theta-10\eta},
\end{equation} 
the condition on $\theta$ in \eqref{mainconds} and $M\le X^{15/22}$ from \eqref{Sitarget}, we deduce that
$$
\mathcal{S}_1^{\ast}\ll \left(X^{1-\theta-10\eta}+X^{15/22+\theta}\right)X^{\varepsilon+\eta} \ll X^{1-2\eta}
$$
if $\varepsilon$ and $\eta$ are suitably small, as desired. 

To estimate $\mathcal{T}_1^{\ast}$, we define and bound $\alpha_k$ as in \eqref{alphakdef}, set $K:=MU$ and $\beta_n:=b_n^{\ast}$ and recall  $N\asymp X/M$ from \eqref{Titarget}, thus obtaining
$$
\mathcal{T}_1^{\ast}\ll (UX)^{1/2} \left(\frac{UX}{q}+MU+\frac{X}{M}+q\right)^{1/2}X^{\varepsilon}.
$$ 
Recalling  \eqref{mainconds}, \eqref{recall} and $X^{7/22}\le M\le X^{8/22}$ from \eqref{Titarget}, we deduce that
$$
\mathcal{T}_1^{\ast}\ll X^{(1+\theta)/2} \left(X^{1-\theta-10\eta}+X^{8/22+\theta}+X^{15/22}\right)^{1/2}X^{\varepsilon+\eta}\ll X^{1-2\eta}
$$
if $\varepsilon$ and $\eta$ are suitably small, as desired.
\end{proof}

\section{Estimations of $\mathcal{S}_2^{\ast}$ and $\mathcal{T}_2^{\ast}$}
Our estimations of $\mathcal{S}_2^{\ast}$ and $\mathcal{T}_2^{\ast}$ follow closely the treatments of similar type I and II sums in \cite{H-B}. A crucial tool in this connection is the following standard estimate for exponential sums.  
\begin{lemma}[van der Corput]\label{lemma_van_bound}
		Let $a$ and $b$ be integers such that $a<b$. Suppose that $f:[a,b]\rightarrow \mathbb{R}$ is a twice continuously differentiable function satisfying 
		\begin{align*}
			f''(t)\asymp \Lambda \quad \text{for all } t\in [a,b],
		\end{align*} 
		where $\Lambda>0$. Then 
		\begin{align*}
			\left|\sum_{a<n\leq b}e(f(n))\right|\ll (b-a)\Lambda^{1/2}+\Lambda^{-1/2}.
		\end{align*}
	\end{lemma}
\begin{proof}
This is \cite[Theorem 2.2]{GrKo}.
\end{proof}
For the estimation of the type II sum $\mathcal{T}_2^{\ast}$, we will use the following bound due to Heath-Brown.

\begin{lemma}[Heath-Brown]\label{Lemma_H-B_bound1}
   Let $\mathcal{T}_2^{\ast}$ be given as in \eqref{T2astdef}. Then, for arbitrary small real numbers $\varepsilon, \eta>0$, we have 
    \begin{align*}
        |\mathcal{T}_2^{\ast}|^2\ll &X^{2(\varepsilon+\eta)}\left(X^{5/2-\gamma}+X^{3-2\gamma}+X^{3-\gamma}N^{-1} +X^{(10-5\gamma)/3}N^{1/3}+X^{(8-4\gamma)/3}N^{2/3}
			+\right. \\ & \left. X^{(10-8\gamma)/3}N^{4/3}\right).
    \end{align*}
\end{lemma}
\begin{proof}
The estimation of $\mathcal{T}_2^{\ast}$ is similar to the estimation of the term $L$ in  \cite[section~4]{H-B}. The final estimate is found in \cite[page 257]{H-B} (with $X$ replaced by $N$, and $N$ replaced by $Y$, respectively).
\end{proof}
Our estimate of the type II sum $\mathcal{T}_2^{\ast}$ is as follows. 
 \begin{lemma}\label{Lemma_H-B_bound}
		There is $\eta>0$ such that $\mathcal{T}_2^{\ast}\ll X^{1-2\eta}$. 
	\end{lemma}
	\begin{proof}
		 Applying Lemma~\ref{Lemma_H-B_bound1}, we have
		\begin{align*}
			\left|\mathcal{T}_2^{\ast}\right|^2\ll & \left(X^{5/2-\gamma}+X^{3-2\gamma}+X^{3-\gamma}N^{-1} +X^{(10-5\gamma)/3}N^{1/3}+X^{(8-4\gamma)/3}N^{2/3}
			+\right. \\ & \left. X^{(10-8\gamma)/3}N^{4/3}\right)X^{2(\varepsilon+\eta)}\\
			\ll & 
\left(X^{5/2-\gamma}+X^{3-2\gamma}+X^{2-\gamma}M+X^{(11-5\gamma)/3}M^{-1/3}+X^{(10-4\gamma)/3}M^{-2/3}+\right. \\ & \left. X^{(14-8\gamma)/3}M^{-4/3}\right)X^{2(\varepsilon+\eta)},
		\end{align*}
		where in the second line we have used that $MN\asymp X$.
		Thus $\mathcal{T}_2^{\ast}\ll X^{1-2\eta}$ if 
		\begin{align} \label{gammacond}
			\gamma>\frac{1}{2}+2\varepsilon+6\eta
		\end{align}
and 
		\begin{align} \label{M_1}
		 X^{5-5\gamma+6\varepsilon+18\eta}\ll M\ll X^{\gamma-2\varepsilon-6\eta}.
		\end{align}
Interchanging the roles of $m$ and $n$, we also have $\mathcal{T}_2^{\ast}\ll X^{1-2\eta}$ under the conditions \eqref{gammacond} and 
$$
X^{5-5\gamma+6\varepsilon+18\eta}\ll N\asymp \frac{X}{M} \ll X^{\gamma-2\varepsilon-6\eta},
$$  
i.e.
		\begin{align}\label{M_2}
			X^{1-\gamma+2\varepsilon+6\eta}\ll M\ll X^{5\gamma-4-6\varepsilon-18\eta}.
		\end{align}
Recalling $X^{7/22}\le M\le X^{8/22}$ from \eqref{Titarget} and $\gamma=1/c>8/9$ from Theorem \ref{main_thm}, and noting that $1-\gamma<7/22<8/22<5\gamma-4$ if  $\gamma>8/9$, the desired bound $\mathcal{T}_2^{\ast}\ll X^{1-2\eta}$ now follows upon taking $\varepsilon$ and $\eta$ sufficiently small.
	\end{proof}
We point out that we have not used \eqref{M_1} but only \eqref{M_2} to establish Lemma \ref{Lemma_H-B_bound}. However, if $b_n^{\ast}=1$, we immediately deduce the following result on $\mathcal{S}_2^{\ast}$ under the condition \eqref{M_1}.

\begin{lemma}[Large $M$] \label{S21} Suppose that $\varepsilon$ and $\eta$ are small enough and $$X^{5-5\gamma+6\varepsilon+18\eta}\ll M\ll X^{\gamma-2\varepsilon-6\eta}.$$ Then $\mathcal{S}_2^{\ast}\ll X^{1-2\eta}$.
\end{lemma} 

This will be useful if $M$ is large. Next, we prove the following for $M$ in a medium range.

\begin{lemma}[Medium $M$] \label{S22} Suppose that $\varepsilon$ and $\eta$ are small enough and 
$$
X^{2-2\gamma+2\varepsilon+6\eta}\ll M\ll X^{\min\{2/3,4\gamma-3-4\varepsilon-12\eta\}}.
$$ 
Then $\mathcal{S}_2^{\ast}\ll X^{1-2\eta}$.
	\end{lemma}
\begin{proof}
Following the estimation of the term $K$ in \cite[section 5]{H-B} and taking the exponent pair $(p,q)=(1/2,1/2)$ in \cite[Lemma 7]{H-B}, we have
\begin{align*}
\mathcal{S}_{2}^{\ast} \ll X^{-\gamma}\left(X^2N^{-1/4}+X^{3/2}N^{1/2}+X^{7/4}N^{1/8}\right)X^{\varepsilon+\eta},
		\end{align*}
provided that $N\ge X^{1/3}$. (Again, $X,N$ replace the variables $N,Y$ in \cite{H-B}, respectively.) Using $MN\asymp X$, it follows that
\begin{align*}
\mathcal{S}_{2}^{\ast} \ll \left(X^{7/4-\gamma}M^{1/4}+X^{2-\gamma}M^{-1/2}+X^{15/8-\gamma}M^{-1/8}\right)X^{\varepsilon+\eta},
		\end{align*}
provided that $M\le X^{2/3}$. 
Now, $$ X^{7/4-\gamma}M^{1/4}X^{\varepsilon+\eta}\ll X^{1-2\eta}\Longleftrightarrow M\ll X^{4\gamma-3-4\varepsilon-12\eta},$$
$$X^{2-\gamma}M^{-1/2}X^{\varepsilon+\eta}\ll X^{1-2\eta}\Longleftrightarrow M\gg X^{2-2\gamma+6\eta+2\varepsilon}$$
and $$ X^{15/8-\gamma}M^{-1/8}X^{\varepsilon+\eta}\ll X^{1-2\eta}\Longleftrightarrow M\gg X^{7-8\gamma+8\varepsilon+24\eta}.$$
Since, $8/9<\gamma<1$ and $\varepsilon,\eta$ are arbitrary small, we have $$7-8\gamma+8\varepsilon+24\eta<2-2\gamma+6\eta+2\varepsilon.$$
This implies the result in Lemma \ref{S22}.
\end{proof}

Finally, for small $M$, we prove the following by a direct appeal to Lemma \ref{lemma_van_bound}.

\begin{lemma}[Small $M$]\label{S23}
		Suppose that $\eta$ is small enough and $M\ll X^{\gamma-1/2-4\eta}$. Then $\mathcal{S}_2^{\ast}\ll X^{1-2\eta}$.
	\end{lemma}
	\begin{proof}
Applying Lemma \ref{lemma_van_bound} with $f(n):=h(mn)^\gamma$ to bound the smooth sum over $n$, and summing over $m$ and $h$ trivially, we obtain 
		\begin{align*}
			\mathcal{S}_2^{\ast} \ll 
			V^{3/2}X^{\gamma/2}M+V^{1/2}X^{1-\gamma/2}\ll
X^{3/2-\gamma+2\eta}M,
\end{align*}
where we have used $V\le H\ll X^{1-\gamma+\eta}$. This implies the result in Lemma \ref{S23}.
	\end{proof}
Since 
$$
(0,\gamma-1/2) \cup (2-2\gamma,\min\{2/3,4\gamma-3\})\cup (5-5\gamma,\gamma)=(0,\gamma)\supset (0,15/22)$$  
if $\gamma>8/9$, Lemmas \ref{S21}, \ref{S22} and \ref{S23} cover all relevant ranges if $\varepsilon$ and $\eta$ are small enough, and we thus have the following.

\begin{lemma}  \label{S2final} There is $\eta>0$ such that  $\mathcal{S}_2^{\ast}\ll X^{1-2\eta}$.\end{lemma}

Here we point out that the condition $\gamma>8/9$ above comes from the inequality $5-5\gamma<4\gamma-3$. 

\section{Estimations of $\mathcal{S}_3^{\ast}$ and $\mathcal{T}_3^{\ast}$}
Our estimations of $\mathcal{S}_3^{\ast}$ and $\mathcal{T}_3^{\ast}$ follow closely the treatments of similar type I and II sums in \cite{Balog}. We first estimate the type II sum $\mathcal{T}_3^{\ast}$.
\begin{lemma} \label{finalT3} There is $\eta>0$ such that $\mathcal{T}_3^{\ast}\ll X^{1-2\eta}$. 
\end{lemma}
	\begin{proof}
		Estimating the triple sum over $m,n,h$ on the right-hand side of  \eqref{T3astdef} similarly as in \cite[inequality (4.4)]{Balog}, and summing over $l$ trivially, we obtain 
		\begin{equation} \label{T3ini}
\begin{split}
			\mathcal{T}_3^{\ast}\ll U & \left( M^{1/2}Q^{1/2}V^{1/2}X^{1/2}+V^{5/4}X^{1+\gamma/4}M^{-1/2}Q^{-1/4}+Q^{1/2}V^{1/2}X^{1-\gamma/2}+ \right. \\ & \left. Q^{1/4}V^{3/4}X^{1-\gamma/4}\right)X^{\varepsilon}
\end{split}
		\end{equation}
if $Q\gg 1$.
		We choose $Q$ in such a way that the first term satisfies $UM^{1/2}Q^{1/2}V^{1/2}X^{1/2+\varepsilon}\ll X^{1-2\eta}$, i.e. 
		\begin{align} \label{Qchoice}
			Q:=\frac{X^{1-2\varepsilon-4\eta}}{U^2MV}.
		\end{align}
Using $U\le L\ll X^{\theta+\eta}$ and $V\le H\ll X^{1-\gamma+\eta}$, the condition $Q\gg 1$ holds if 
\begin{equation} \label{firstMcond}
M\ll X^{\gamma-2\theta-2\varepsilon-7\eta}.
\end{equation}
Plugging \eqref{Qchoice} into \eqref{T3ini}, and using $U\ll X^{\theta+\eta}$ and $V\ll X^{1-\gamma+\eta}$ again, we get 
		\begin{align*}
			\mathcal{T}_3^{\ast} \ll & X^{1-2\eta}+\left(U^{3/2}V^{3/2}X^{3/4+\gamma/4}M^{-1/4}+X^{3/2-\gamma/2}M^{-1/2}+U^{1/2}V^{1/2}X^{5/4-\gamma/4}M^{-1/4}\right)X^{2\epsilon+\eta}\\
\ll  &X^{1-2\eta}+\left(X^{9/4+3\theta/2-5\gamma/4}M^{-1/4}+X^{3/2-\gamma/2}M^{-1/2}+X^{7/4+\theta/2-3\gamma/4}M^{-1/4}\right)X^{2\varepsilon+3\eta}.
		\end{align*} 
The last line is $\ll X^{1-2\eta}$ if 
\begin{equation} \label{secondMcond}
M\gg X^{5-5\gamma+6\theta+8\varepsilon+20\eta}.
\end{equation}
Combining \eqref{firstMcond} and \eqref{secondMcond}, the desired bound $\mathcal{T}_3^{\ast}\ll X^{1-2\eta}$ holds provided that 
		\begin{align}\label{M_11}
			X^{5-5\gamma+6\theta+8\varepsilon+20\eta}\ll M\ll X^{\gamma-2\theta-2\varepsilon-7\eta}.
		\end{align}
		Interchanging the roles of $m$ and $n$, we also have $\mathcal{T}_3^{\ast}\ll X^{1-2\eta}$ when
$$
X^{5-5\gamma+6\theta+8\varepsilon+20\eta}\ll N\asymp \frac{X}{M} \ll X^{\gamma-2\theta-2\varepsilon-7\eta},
$$ 
i.e.
		\begin{align}\label{M_12}
			X^{1-\gamma+2\theta+2\varepsilon+7\eta}\ll M\ll X^{5\gamma-4-6\theta-8\varepsilon-20\eta}.
		\end{align}
Recalling our condition $X^{7/22}\le M\le X^{8/22}$ from \eqref{Titarget} and observing that  
\begin{align*}
1-\gamma+2\theta<\frac{7}{22}<\frac{8}{22}<5\gamma-4-6\theta
		\end{align*}
under the conditions in \eqref{mainconds}, the desired bound $\mathcal{T}_3^{\ast}\ll X^{1-2\eta}$ now follows upon taking $\varepsilon$ and $\eta$ sufficiently small.	\end{proof}

We point out that we have not used \eqref{M_11} but only \eqref{M_12} to establish Lemma~\ref{finalT3}. However, if $b_n^{\ast}=1$, we immediately deduce the following result on $\mathcal{S}_3^{\ast}$ under the condition \eqref{M_11}.

\begin{lemma}[Large $M$] \label{S31} Suppose that $\varepsilon$ and $\eta$ are small enough and $$X^{5-5\gamma+6\theta+8\varepsilon+20\eta}\ll M\ll X^{\gamma-2\theta-2\varepsilon-7\eta}.$$ Then $\mathcal{S}_3^{\ast}\ll X^{1-2\eta}$.
\end{lemma} 

This will be useful if $M$ is large. Next, we prove the following for $M$ in a medium range.

\begin{lemma}[Medium $M$] \label{S32} Suppose that $\varepsilon$ and $\eta$ are small enough and 
$$
X^{7-8\gamma+8\theta+8\varepsilon+40\eta}\ll M\ll X^{4\gamma-3-4\theta-4\varepsilon-20\eta}.
$$ 
Then $\mathcal{S}_3^{\ast}\ll X^{1-2\eta}$.
	\end{lemma}

\begin{proof}
		Estimating the triple sum over $m,n,h$ on the right-hand side of  \eqref{S3astdef} similarly as in \cite[last inequality on page 60]{Balog}, and summing over $l$ trivially, we obtain 
\begin{equation*}
			\mathcal{S}_3^{\ast}\ll U \left(V^{5/4}M^{1/4}X^{1/2+\gamma/4}+V^{7/8}M^{-1/8}X^{1-\gamma/8}\right)X^{\varepsilon}.
\end{equation*}
Using $U\ll X^{\theta+\eta}$ and $V\ll X^{1-\gamma+\eta}$, it follows that
		\begin{align*}
			\mathcal{S}_3^{\ast} \ll \left(X^{7/4+\theta-\gamma}M^{1/4}+X^{15/8+\theta-\gamma}M^{-1/8}\right)X^{\varepsilon+3\eta}.
		\end{align*} 
This implies the result in Lemma \ref{S32}.
	\end{proof}

Finally, for small $M$, we prove the following by a direct appeal to Lemma \ref{lemma_van_bound}.

\begin{lemma}[Small $M$] \label{S33} Suppose that $\eta$ is small enough and $M\ll X^{\gamma-1/2-\theta-5\eta}$. Then $\mathcal{S}_3^{\ast}\ll X^{1-2\eta}$.
	\end{lemma}

\begin{proof} Applying Lemma \ref{lemma_van_bound} with $f(n):=h(mn)^\gamma$ to bound the smooth sum over $n$, and summing over $m$, $l$ and $h$ trivially, we obtain 
		\begin{align*}
			\mathcal{S}_3^{\ast} \ll U\left( 
			V^{3/2}X^{\gamma/2}M+V^{1/2}X^{1-\gamma/2}\right)\ll
X^{3/2-\gamma+\theta+3\eta}M,
\end{align*}
where we have used $U\ll X^{\theta+\eta}$ and $V\ll X^{1-\gamma+\eta}$. This implies the result in Lemma \ref{S33}. 		
		\end{proof}

Since 
$$(0,\gamma-1/2-\theta) \cup (7-8\gamma+8\theta,4\gamma-3-4\theta)\cup (5-5\gamma+6\theta,\gamma-2\theta)=(0,\gamma-2\theta)\supset (0,15/22)
$$  
under the conditions in \eqref{mainconds}, Lemmas \ref{S31}, \ref{S32} and \ref{S33} cover all relevant ranges if $\varepsilon$ and $\eta$ are small enough, and we thus have the following, completing the proof of Theorem \ref{main_thm}.

\begin{lemma}  \label{S3final} There is $\eta>0$ such that  $\mathcal{S}_3^{\ast}\ll X^{1-2\eta}$.
\end{lemma}

Here we point out that the condition $\theta<(9\gamma-8)/10$ in \eqref{mainconds} comes from the inequality $5-5\gamma+6\theta<4\gamma-3-4\theta$. We also note that $\mathcal{S}_2^{\ast}$ and $\mathcal{T}_2^{\ast}$ could have been bounded in a similar way as $\mathcal{S}_3^{\ast}$ and $\mathcal{T}_3^{\ast}$ above since the method of Balog and Friedlander also  
applies to the case when $l=0$. However, we decided to keep our separate treatment of $\mathcal{S}_2^{\ast}$ and $\mathcal{T}_2^{\ast}$ because the condition $\gamma>8/9$ emerges most naturally from it, and these terms are independent of $\theta$.  

\end{document}